\documentclass[11pt]{amsart}
\usepackage{latexsym,amssymb,amsmath,epsfig}
\usepackage{graphicx}
\newtheorem{theorem}{Theorem}[section]
\theoremstyle{plain}

\newtheorem{corollary}{Corollary}[section]

\newtheorem{example}{Example}[section]

\newtheorem{lemma}{Lemma}[section]

\newtheorem{proposition}{Proposition}[section]

\numberwithin{equation}{section}

\begin{document}

\title{New constructions for the $n$-queens problem}

\author{M. Ba\v{c}a}
\address{Department of Applied Mathematics and Informatics \\
Technical University in Ko¨ice\\
Letn\'a 9, 
Ko\v{s}ice, Slovakia}
\email{ martin.baca@tuke.sk}

\author{S. C. L\'opez}
\address{%
Departament de Matem\`{a}tica\\
Universitat de Lleida\\
C/Jaume II, 69\\
25001 Lleida, Spain}
\email{susana.clara.lopez@matematica.udl.cat}

\author{F. A. Muntaner-Batle}
\address{Graph Theory and Applications Research Group \\
 School of Electrical Engineering and Computer Science\\
Faculty of Engineering and Built Environment\\
The University of Newcastle\\
NSW 2308
Australia}
\email{famb1es@yahoo.es}

\author{A. Semani\v{c}ov\'{a}-Fe\v{n}ov\v{c}\'ikov\'a }
\address{Department of Applied Mathematics and Informatics \\
Technical University in Ko\v{s}ice\\
Letn\'a 9, 
Ko\v{s}ice, Slovakia}
\email{andrea.fenovcikova@tuke.sk}

\maketitle

\begin{abstract}
Let $D$ be a digraph, possibly with loops. A queen labeling of $D$ is a bijective function $l:V(G)\longrightarrow \{1,2,\ldots,|V(G)|\}$ such that, for every pair of arcs in $E(D)$, namely $(u,v)$ and $(u',v')$ we have (i) $l(u)+l(v)\neq l(u')+l(v')$ and (ii) $l(v)-l(u)\neq l(v')-l(u')$. Similarly, if the two conditions are satisfied modulo $n=|V(G)|$, we define a modular queen labeling. There is a bijection between (modular) queen labelings of $1$-regular digraphs and the solutions of the (modular) $n$-queens problem.

The $\otimes_h$-product was introduced in 2008 as a generalization of the Kronecker product and since then, many relations among labelings have been established using the $\otimes_h$-product and some particular families of graphs.

In this paper, we study some families of $1$-regular digraphs that admit (modular) queen labelings and present a new construction concerning to the (modular) $n$-queens problem in terms of the $\otimes_h$-product, which in some sense complements a previous result due to P\'olya.

\end{abstract}

\begin{quotation}

\end{quotation}


\thispagestyle{empty}

\section{Introduction}
\label{intro}
The {\it $n$-queens problem} consists in placing $n$ nonattacking queens on an $n\times n$ chessboard. The modular version of this problem, namely the {\it modular $n$-queens problem} consists in placing $n$ nonattacking queens on an $n\times n$ chessboard where the opposite sides are identified to make a torus. It is known that there is at least one solution to the $n$-queens problem for all $n\ge 4$ on a $n\times n$ standard chessboard \cite{Pau1,Pau2}. However, for the modular version of the problem, a solution exists if and only if, $\gcd(n,6)=1$ \cite{Polya}.
See \cite{BellSte} for a survey of other known results and research going on for this problem, as well as a history of it. 

For the graph theory terminology and notation not defined in this paper we refer the reader to one of the following sources \cite{BaMi,CH,G,Wa}.
However, in order to make this paper reasonably self-contained, we mention that by a $(p,q)$-graph we mean a graph of $p$ vertices and $q$ edges. For integers $m\le n$, we denote $[m,n]=\{m,m+1,\ldots,n\}$.

Let $D$ be a digraph, possibly with loops.
A {\it queen labeling} \cite{BloLamMunRiu} of $D$ is a bijective function $l:V(D)\longrightarrow [1,|V(D)|]$ such that, for every pair of arcs in $E(D)$, namely $(u,v)$ and $(u',v')$ we have (i) $l(u)+l(v)\neq l(u')+l(v')$ and (ii) $l(v)-l(u)\neq l(v')-l(u')$. It is well known that there is a bijection between the set of queen labelings of $1$-regular digraphs of order $n$ and the set of solutions of the $n$-queens problem. The equivalence between $n$-queens solution and these labelings is based on the following observation. Assume that all rows and columns of the $n\times n$ chessboard are labeled from up to down and from left to right, respectively, according to the position. If we assign to each cell $(i,j)$, (where the first component is the row and the second the column), the sum $s(i,j)=i+j$ and the difference $d(i,j)=j-i$, then, all cells that belong to the same secondary diagonal have the same sum and all cells that belong to the same primary diagonal have the same difference (see more details in \cite{BellSte,BloLamMunRiu}). Similarly, we can introduce a {\it modular queen labeling} of a digraph $D$. A {\it modular queen labeling} of a digraph $D$ is a bijective function $l:V(D)\longrightarrow [1,|V(D)|]$ such that, for every pair of arcs in $E(D)$, namely $(u,v)$ and $(u',v')$ we have (i) $l(u)+l(v)\neq l(u')+l(v')$ (mod $n$) and (ii) $l(v)-l(u)\neq l(v')-l(u')$ (mod $n$). Thus, if a digraph admits a modular queen labeling then $|E(D)|\le |V(D)|$. In particular, a modular queen labeling of a $1$-regular digraph $D$ is a harmonious labeling $l$ of und$(D)$, where und$(D)$ means the underlying graph of $D$, satisfying the extra condition that for every pair of arcs $(u,v)$ and $(u',v')$ in $E(D)$, we have  $l(v)-l(u)\neq l(v')-l(u')$ $\pmod n$. A digraph that admits a (modular) queen labeling is called a ({\it modular})  {\it queen digraph}.

Kotzig and Rosa \cite{K1} introduced in 1970, the concepts of edge-magic graphs and edge-magic labelings as follows: Let $G$ be a $(p,q)$-graph. Then $G$ is called {\it edge-magic} if there is a bijective function $f:V(G)\cup E(G)\longrightarrow [1,p+q]$ such that the sum $f(x)+f(xy)+f(y)=k$ for any $xy\in E(G)$. Such a function is called an {\it edge-magic labeling} of $G$ and $k$ is called the {\it valence} \cite{K1} or the {\it magic sum} \cite{Wa} of the labeling $f$. Motivated by the concept of edge-magic labelings, Enomoto et al. \cite{E} introduced in 1998 the concepts of super edge-magic graphs and labelings as follows: Let $f:V(G)\cup E(G) \longrightarrow[1,p+q]$ be an edge-magic labeling of a $(p,q)$-graph G with the extra property that $f(V(G))=[1,p].$ Then G is called { \it super edge-magic} and $f$ is a {\it super edge-magic labeling} of $G$. It is worthwhile mentioning that Acharya and Hegde had already defined in \cite{AH} the concept of strongly indexable graph that turns out to be equivalent to the concept of super edge-magic graph. We take this opportunity to mention that although the original definitions of (super) edge-magic graphs and labelings were provided for simple graphs (that is to say, graphs with no loops nor multiple edges), in this paper, we understand these definitions for any graph. Therefore, unless otherwise specified, the graphs considered in this paper are not necessarily simple. Moreover, we say that a digraph is (super) edge-magic if its underlying graph is (super) edge-magic. In \cite{F2}, Figueroa-Centeno et al. provided the following useful characterization of super edge-magic simple graphs, that works in exactly the same way for graphs in general.

\begin{lemma}\label{super_consecutive} \cite{F2}
Let $G$ be a $(p,q)$-graph. Then $G$ is super edge-magic if and only if  there is a
bijective function $g:V(G)\longrightarrow [1,p]$ such
that the set $S=\{g(u)+g(v):uv\in E(G)\}$ is a set of $q$
consecutive integers.
In this case, $g$ can be extended to a super edge-magic labeling $f$ with valence $p+q+\min S$.
\end{lemma}
Figueroa et al. defined in \cite{F1}, the following product: Let $D$ be a digraph and let $\Gamma$ be a family of digraphs with the same vertex set $V$. Assume that $h: E(D) \longrightarrow \Gamma$ is any function that assigns elements of $\Gamma$ to the arcs of $D$. Then the digraph $D \otimes _{h} \Gamma $ is defined by (i) $V(D \otimes _{h} \Gamma)= V(D) \times V$ and (ii) $((a,i),(b,j)) \in E(D \otimes _{h} \Gamma) \Leftrightarrow (a,b) \in E(D)$ and $(i,j) \in E(h(a,b))$. Note that when $h$ is constant, $D \otimes _{h} \Gamma$ is the Kronecker product. Many relations among labelings have been established using the $\otimes_h$-product and some particular families of graphs, namely $\mathcal{S}_p$ and $\mathcal{S}_p^k$ (see for instance, \cite{ILMR,LopMunRiu1,LopMunRiu6,PEM_LMR2,MJM_LMR,SlMb,LMP1}).
The family $\mathcal{S}_p$ contains all super edge-magic $1$-regular labeled digraphs of order $p$ where each vertex takes the name of the label that has been assigned to it. A super edge-magic digraph $F$ is in $\mathcal{S}_p^k$ if $|V(F)|= |E(F)|=p$ and the minimum sum of the labels of the adjacent vertices (see Lemma \ref{super_consecutive}) is equal to $k$. Notice that, since each $1$-regular digraph has minimum edge induced sum equal to $(p+3)/2$, it follows that $ \mathcal{S}_p \subset \mathcal{S}_p^{(p+3)/2}$. The following result was introduced in \cite{LopMunRiu6}, generalizing a previous result found in \cite{F1}:

\begin{theorem} \label{spk} \cite{LopMunRiu6}
Let $D$ be a (super) edge-magic digraph and let $h: E(D) \longrightarrow \mathcal{S}_p^k$ be any function. Then $D\otimes _{h} \mathcal{S}_p^k$ is (super) edge-magic.
\end{theorem}

Let $Q(n)$ denote the number of solutions for the $n$-queens problem and $M(n)$ the number of solutions for the modular $n$-queens problem.

\begin{theorem}\cite{RivVarZim} For all $m$ and $n$ for which $m\ge 3$ and $\gcd(n, 6)$ = 1, it holds that $Q(mn)>(Q(m))^nM(n)$. In particular, if $\gcd(N,30)=5$ then $Q(N)>4^{N/5}$.
\end{theorem}

The first part of the above theorem is a corollary of a previous result due to P\'olya, where a solution of the $n$-queens problem is described by means of a bijective function $f: [0,n-1]\longrightarrow [0,n-1]$, so that the $k^{th}$ queen is placed at the $(k,f(k))$ coordinate of the chessboard.

\begin{theorem}\cite{Polya}\label{theo: Polya}
For given $m,n > 3$ such that $\gcd(n,6) = 1$, $f_1,..., f_{Q(m)}$ are all solutions for the $m\times m$ standard board, and $g$ is a solution for the $n\times n$ modular board, then for each map $\pi :[0,n-1]\longrightarrow [1,Q(m)]$ the function $h(an + b) = f_{\pi (b)}(a)n + g(b)$ gives a distinct $mn \times mn$ solution for the standard board.
\end{theorem}
In this paper,  we present a new construction concerning to the (modular) $n$-queens problem in terms of the $\otimes_h$-product, which in some sense complements a previous result due to P\'olya. We also study some families of $1$-regular digraphs that admit (modular) queen labelings. 

\section{Queen digraphs and the $\otimes_h$ product}
We start this section by introducing some notation. For the rest of the paper, whenever we work with a (modular) queen digraph, we will assume that the vertices  are identified with the labels of a (modular) queen labeling.

Let $D$ be a queen digraph and assume $(u,v)\in E(D)$. We define $s(u,v)=u+v$ and $s(D)=\{s(u,v), \ (u,v)\in E(D)\}$. Similarly, for any $(u,v)\in E(D)$, we define $d(u,v)=v-u$ and $d(D)=\{d(u,v), \ (u,v)\in E(D)\}$.

The next two lemmas are trivial.

\begin{lemma}\label{lemma: S(F) and d(F)}
Let $F$ be a queen digraph of order $n$. Then,
\begin{itemize}
\item[(i)] $s(F)\subset [2,2n]$ and
\item[(ii)] $d(F)\subset [-(n-1),n-1]$.
\end{itemize} 
\end{lemma}

Let $J$ be a finite set of integers. We denote by $\Sigma (J)$ the sum of all integers in $J$. 

\begin{lemma}\label{lemma: sigma S(F) and sigma d(F)}
Let $F$ be a $1$-regular queen digraph of order $n$. Then,
$\Sigma (s(F))=n(n+1)$ and
 $\Sigma (d(F))=0$.
 \end{lemma}
\begin{proof}
When considering $\Sigma (s(F))$, we sum all integers in $[1,n]$ twice, while when we calculate $\Sigma (d(F))$ every integer in $[1,n]$ appears also twice but with opposite sign. \qed
\end{proof}

Let $J$ be a subset of integers. We denote by $J-n=\{j-n, \ j\in J\}$.
\begin{theorem}\label{theo: queen_and_product}
Let $D$ be any queen digraph. Let $\Gamma$ be a family of queen digraphs of order $n$ and let $h:E(D)\longrightarrow \Gamma$ be any function such that for any pair of arcs $(u,v)$ and $(u',v')$ in $E(D)$ the following conditions hold
\begin{itemize}
\item[(i)]  If $s(u,v)=s(u',v')-1$, then $(s(h(u,v))-n)\cap s(h(u',v'))=\emptyset$.
\item[(ii)]  If $d(u,v)=d(u',v')-1$, then $(d(h(u,v))-n)\cap d(h(u',v'))=\emptyset$.
\end{itemize}
Then,
  $D\otimes_h\Gamma$ is a queen digraph.
 \end{theorem}
\begin{proof}
Recall that we assume that the vertices of $D$ and each element of $\Gamma$ is renamed after the labels of their corresponding queen labeling. Consider, similarly to what was done in the proof of Theorem \ref{spk}, the labeling $l$ of the product $D\otimes _{h}\Gamma$ that assigns to the vertex $(a,i) \in V(D\otimes _{h} \Gamma)$ the label: $n(a-1)+i$. Thus, an arc $((a,i),(b,j)) \in E(D\otimes _{h}\Gamma)$ has:
\begin{equation}\label{eq: sum arc product}
s((a,i),(b,j))=n(a+b-2)+i+j,\end{equation}
and 
\begin{equation}\label{eq: difference arc product}
 d((a,i),(b,j))=n(b-a)+j-i.\end{equation}

 Let us check now that, for any two pairs of different arcs $((a,i),(b,j))$ and $((a',i'),(b',j'))$ in $E(D\otimes _{h}\Gamma)$, (c1) $s((a,i),(b,j))\neq s((a',i'),(b',j'))$ and (c2) $d((a,i),(b,j))\neq d((a',i'),(b',j'))$. 
We start by proving (c1). Without loss of restriction assume that $s(a',b')>s(a,b)$. Notice that, if $s(a',b')\ge s(a,b)+ 2$ then, by (\ref{eq: sum arc product}) and Lemma \ref{lemma: S(F) and d(F)}(i), we obtain 
\begin{eqnarray*}
s((a',i'),(b',j'))&=n(a'+b'-2)+i'+j'\ge  
n(a+b+2-2)+2\\
&>n(a+b-2)+2n\ge s((a,i),(b,j)).
\end{eqnarray*}

Thus, assume that $s(a',b')= s(a,b)+ 1$. Hence, the value for $s((a',i'),(b',j'))$ is
$$n(a'+b'-2)+i'+j'=n(a+b+1-2)+i'+j'=n(a+b-2)+(i'+j'+n),$$
which clearly differs from $s((a,i),(b,j))=n(a+b-2)+(i+j)$ by condition (i) in the statement of the theorem together with the facts that $i+j\in s(h(a,b))$ and $i'+j'\in s(h(a',b'))$. The proof of (c2) is similar. Without loss of restriction assume that $d(a',b')>d(a,b)$. Notice that, if $d(a',b')\ge d(a,b)+ 2$ then, by (\ref{eq: difference arc product}) and Lemma \ref{lemma: S(F) and d(F)}(ii), we will obtain 
\begin{eqnarray*}d((a',i'),(b',j'))&=n(b'-a')+j'-i'\ge  n(b-a+2)-(n-1)\\
&>n(b-a)+n-1\ge d((a,i),(b,j)).\end{eqnarray*}
Thus, assume that $d(a',b')= d(a,b)+ 1$. Hence, the value for $d((a',i'),(b',j'))$ is
$$n(b'-a')+j'-i'=n(b-a+1)+j'-i'=n(b-a)+(j'-i'+n),$$
which clearly differs from $d((a,i),(b,j))=n(b-a)+(i+j)$ by condition (ii) in the statement of the theorem together with the facts that $j-i\in d(h(a,b))$ and $j'-i'\in d(h(a',b'))$. \qed
\end{proof}

The next corollary gives a sufficient condition to simplify the statement of Theorem \ref{theo: queen_and_product}.
\begin{corollary}
Let $D$ be a queen digraph. Let $\Gamma$ be a family of queen digraphs of order $n$ and let $h:E(D)\longrightarrow \Gamma$ be any function. Let $I$ and $J$ be sets of $n$ integers such that $(I-n)\cap I=(J-n)\cap J=\emptyset$. If, for any $F\in \Gamma$, $s(F)=I$ and $d(F)=J$ then
  $D\otimes_h\Gamma$ is a queen digraph.
\end{corollary}
\begin{proof}
The assumptions $s(F)=I$ and $d(F)=J$ imply conditions (i) and (ii) of Theorem \ref{theo: queen_and_product}.\qed
\end{proof}

The following result gives an application of the $\otimes_h$-product to the $n$-queens problem in terms of queen labelings of $1$-regular digraphs. This result can be thought as a complementary result of Theorem \ref{theo: Polya}, since the solutions that are provided, never appear in the construction provided in Theorem \ref{theo: Polya}.

\begin{corollary}\label{coro: queen_and_product_of_$1$-regular}
Let $D$ be any $1$-regular queen digraph. Let $\Gamma$ be a family of $1$-regular queen digraphs of order $n$ and let $h:E(D)\longrightarrow \Gamma$ be any function such that for any pair of arcs $(u,v)$ and $(u',v')$ in $E(D)$ the following conditions hold.
\begin{itemize}
\item[(i)]  If $s(u,v)=s(u',v')-1$, then $(s(h(u,v))-n)\cap s(h(u',v'))=\emptyset$.
\item[(ii)]  If $d(u,v)=d(u',v')-1$, then $(d(h(u,v))-n)\cap d(h(u',v'))=\emptyset$.
\end{itemize}
Then,
  $D\otimes_h\Gamma$ is a $1$-regular queen digraph.
 \end{corollary}
\begin{proof}
This is an immediate consequence of Theorem \ref{theo: queen_and_product} and the fact that, by definition of the $\otimes_h$-product, the product of $1$-regular digraphs produces $1$-regular digraphs.\qed
\end{proof} 

\begin{example}\label{Example_1}
Let $D$ be the $1$-regular queen labeled digraph defined by $V(D)=[1,5]$ and $E(D)=\{(1,5), (2,3), (3,1),(4,4), (5,2)\}$ and $\Gamma =\{F_1,F_2\}$, where $F_1=D$ and $F_2$ is the digraph obtained from $D$ by reversing all its arcs. Consider $h:E(D)\longrightarrow \Gamma$ defined by: $h(1,5)=h(2,3)=h(4,4)=F_1$ and $h(3,1)=h(5,2)=F_2$. Then, a queen labeling of $D\otimes_h\Gamma$ appears in Fig. \ref{Figure_1}. Notice that the adjacency matrix of $D\otimes_h\Gamma$, where the vertices are identified with the labels assigned by the queen labeling, gives a solution of the $25$-queens problem (see Fig. \ref{Figure_2}).

\begin{figure}[ht]
\begin{center}
\includegraphics[width=340pt]{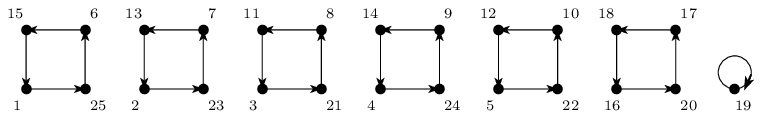}\\

 \caption{A  queen labeled digraph of order 25.}\label{Figure_1}
\end{center}
\end{figure}
\begin{figure}[ht]
\begin{center}
  \includegraphics[width=240pt]{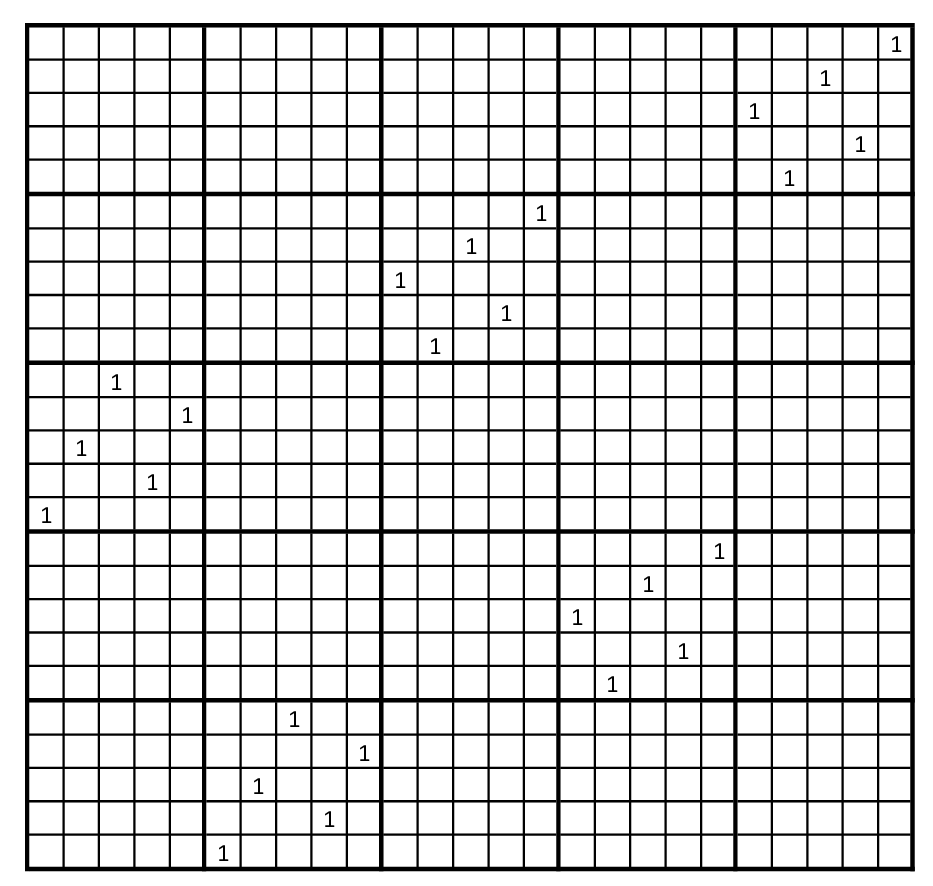}\\ 
 \caption{The solution of the $25$-queens problem obtained from Fig. \ref{Figure_1}.}\label{Figure_2}
\end{center}
\end{figure}
\end{example}

\begin{example}\label{Example_2}
Let $D$ be the $1$-regular queen labeled digraph defined by $V(D)=[1,4]$ and $E(D)=\{(1,3), (2,1), (3,4),(4,2)\}$. Let $\Gamma =\{F_1,F_2,F_3,F_4\}$, where $V(F_i)=[1,8]$, $E(F_1)=\{(1,5), (2,2), (3,4),\\ (4,7), (5,3),(6,8),(7,6), (8,1)\}$,  $E(F_2)=\{(1,1), (2,5), (3,8), (4,6), (5,3),(6,7),(7,2), (8,4)\}$ and the adjacency matrices of $F_3$ and $F_4$ are obtained from the adjacency matrix of $F_2$ by rotating it $\pi/2$ and $\pi$ radiants clockwise, respectively. Consider $h:E(D)\longrightarrow \Gamma$ defined by: $h(2,1)=F_1$, $h(1,3)=F_2$, $h(3,4)=F_3$  and $h(4,2)=F_4$. Then, the adjacency matrix of $D\otimes_h\Gamma$, where the vertices are identified with the labels induced by the product (that is, a vertex $(a,i)$ of the $\otimes_h$-product  is identified with $8(a-1)+i$), gives a solution of the $32$-queens problem (see Fig. \ref{Figure_3}).

Notice that, since $\gcd(4,6)=\gcd(8,6)\neq 1$, no one of the queen solutions related to the digraphs that are involved in the $\otimes_h$-product are modular solutions.
\begin{figure}[ht]

\begin{center}
 \includegraphics[width=307pt]{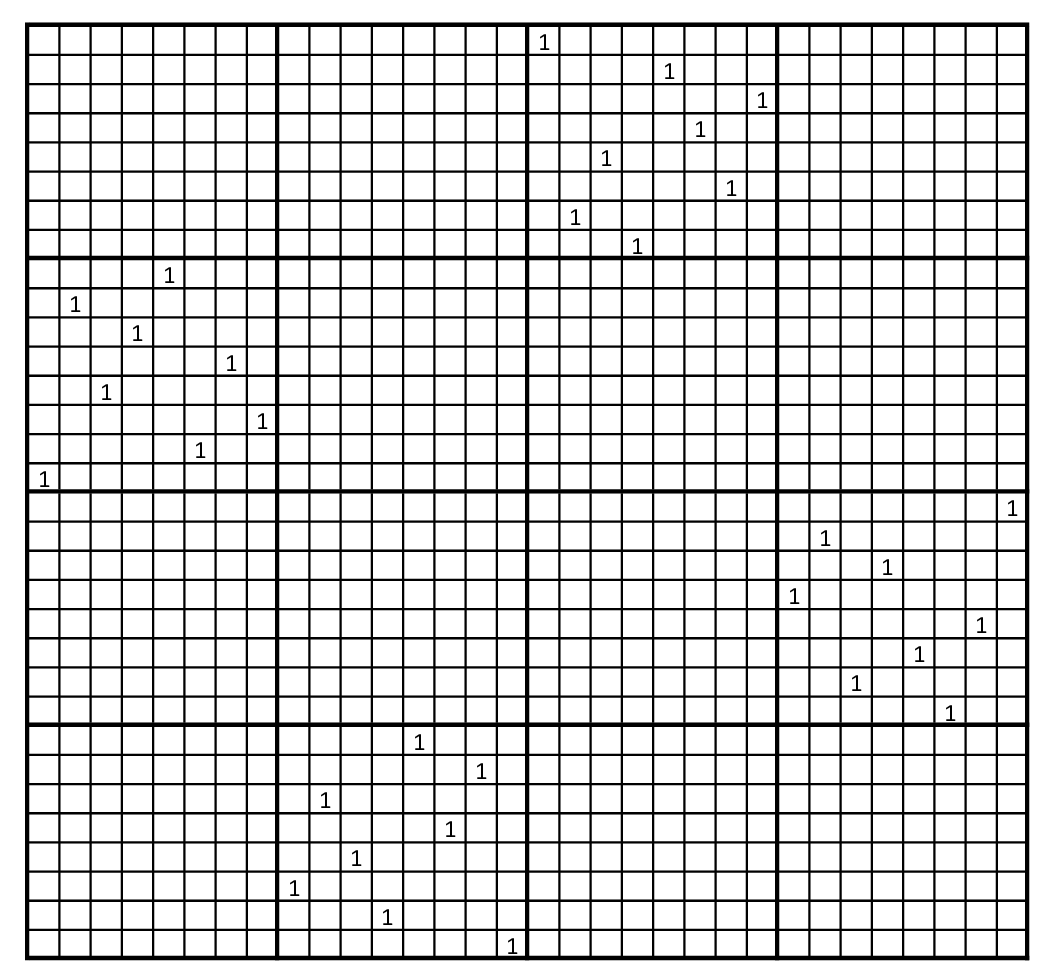}\\ 

 \caption{A solution of the $32$-queens problem obtained from the $\otimes_h$-product.}\label{Figure_3}
 \end{center}
\end{figure}
\end{example}
\section{Modular queen digraphs}
We start this section by applying the $\otimes_h$-product to modular queen labelings. The idea of this proof is related to the one of Theorem 2.1 in \cite{ILMR}.
\begin{theorem}\label{theo: generalization_queen_and_product}
Let $D$ be any modular queen digraph of order $m$. Let $\Gamma$ be a family of modular queen digraphs of order $n$ and let $h:E(D)\longrightarrow \Gamma$ be any function such that for any pair of arcs $(u,v)$ and $(u',v')$ in $E(D)$ the following conditions hold.
\begin{itemize}
\item[(i)]  If $s(u,v)=s(u',v')-1\pmod m$, then $(s(h(u,v))-n)\cap s(h(u',v'))=\emptyset$.
\item[(ii)]  If $d(u,v)=d(u',v')-1\pmod m$, then $(d(h(u,v))-n)\cap d(h(u',v'))=\emptyset$.
\end{itemize}
Then,
  $D\otimes_h\Gamma$ is a modular queen digraph.
 \end{theorem}
\begin{proof}
As in the previous section, we rename the vertices of $D$ and each element in $\Gamma$ after the labels of their corresponding modular queen labelings and, we consider the labeling $l:V(D\otimes_h\Gamma)\longrightarrow [1,mn]$ that assigns to the vertex $(a,i)$ the label $n(a-1)+i$. Let us check that $l$ is a modular queen labeling of $D\otimes_h\Gamma$.
That is, for any two pairs of different arcs $((a,i),(b,j)), ((a',i'),(b',j')) \in E(D\otimes _{h}\Gamma)$, (c1) $s((a,i),(b,j))\neq s((a',i'),(b',j')) \pmod {mn}$ and (c2) $d((a,i),(b,j))\neq d((a',i'),(b',j')) $ (mod $ {mn}$), where $s((a,i),(b,j))$ and $d((a,i),(b,j))$ are the induced sum and the induced difference introduced in (\ref{eq: sum arc product}) and (\ref{eq: difference arc product}), respectively. Since each $F\in\Gamma$ is labeled with a modular queen labeling, all the sums $i+j \pmod n$, for $(i,j)\in F$ are different. The same happens with all the sums $a+b\pmod m$, for $(a,b)\in E(D)$. Thus, (c1) clearly holds whenever $(a,b)=(a',b')$ or $i+j\neq i'+j' \pmod n$. Assume to the contrary that $(a,b)\neq (a',b')$, $i+j=i'+j' \pmod n$ and that $s((a,i),(b,j))= s((a',i'),(b',j')) \pmod {mn}$. By Lemma \ref{lemma: S(F) and d(F)}, either $i+j=i'+j'$ or $i+j-(i'+j')\in \{\pm n\}$.  If $i+j=i'+j'$ then equality $s((a,i),(b,j))= s((a',i'),(b',j')) \pmod {mn}$ implies that $a+b=a'+b'\pmod m$ a contradiction. Hence, $i+j-(i'+j')\in \{-n,n\}$. Without loss of generality, assume $i+j-(i'+j')=n$. Again by equality $s((a,i),(b,j))= s((a',i'),(b',j')) \pmod {mn}$, we conclude that $a+b=a'+b'-1 \pmod m$, which, together with $i+j-(i'+j')=n$, contradicts statement (i) of the theorem.
The proof of (c2) is similar. \qed
\end{proof}

Similarly to what happens with Corollary \ref{coro: queen_and_product_of_$1$-regular}, we can obtain an application of the above result to the modular $n$-queens problem.

\begin{corollary}\label{coro: $1$-regular_modular_queen_and_product}
Let $D$ be any modular queen $1$-regular digraph of order $m$. Let $\Gamma$ be a family of modular queen  $1$-regular digraphs of order $n$ and let $h:E(D)\longrightarrow \Gamma$ be any function such that for any pair of arcs $(u,v)$ and $(u',v')$ in $E(D)$ the following conditions hold.
\begin{itemize}
\item[(i)]  If $s(u,v)=s(u',v')-1\pmod m$, then $(s(h(u,v))-n)\cap s(h(u',v'))=\emptyset$.
\item[(ii)]  If $d(u,v)=d(u',v')-1\pmod m$, then $(d(h(u,v))-n)\cap d(h(u',v'))=\emptyset$.
\end{itemize}
Then,
  $D\otimes_h\Gamma$ is a modular queen  $1$-regular digraph.
 \end{corollary}

\begin{example}\label{Example_3}
It turns out that the labelings of the digraphs that appear in Example \ref{Example_1}, namely $D$, $F_1$ and $F_2$ are modular queens, however condition (ii) in Corollary  \ref{coro: $1$-regular_modular_queen_and_product} does not hold and the resulting labeled digraph obtained by means of the $\otimes_h$-product is not a modular queen labeling, for instance, the arcs $(25,6)$ and $(7,13)$ induce the same difference modulo 25. Therefore, the solution of Fig. \ref{Figure_2}, is not a solution for the modular $25$-queens problem. 
\end{example}

Notice that conditions (i) and (ii)  in Corollary  \ref{coro: $1$-regular_modular_queen_and_product} trivially hold when $h$ is a constant function, due to the modularity of every element in the family $\Gamma$. 
The next result is a direct consequence of this remark.

\begin{corollary}\label{coro: number of modular queens} For all integers $m$ and $n$, it holds that $$M(mn)\ge M(m)M(n).$$ \end{corollary}


\section{Queen labelings of $1$-regular digraphs}
We know that there is a bijection between solutions of the (modular) $n$-queens problem and (modular) queen labelings of $1$-regular digraph of order $n$. In this section we will provide some families of $1$-regular digraphs that admit (modular) queen labelings.

Let $C^+_n$ be a strong orientation of the cycle of $n$ vertices. By checking by hand, we obtain the following information:
\begin{itemize}
\item $n=4$: $C^+_4$ is a queen digraph, but $C^+_3\cup C^+_1$ is not a queen digraph.
\item $n=5$: $C^+_5$ is not a queen digraph, but $C^+_4\cup C^+_1$ is a queen digraph.
\item $n=6$: $C^+_6$ and $2C^+_3$ are queen digraphs, but $C^+_5\cup C^+_1$ is not a queen digraph.
\item $n=7$: $C^+_6\cup C^+_1$, $2C^+_3\cup C^+_1$ and $C^+_7$ are queen digraphs, but $C^+_4\cup C^+_3$ is not a queen digraph.
\end{itemize}

One of the first families that we provide in this paper is obtained by using the modular solution provided by P\'olya in \cite{Polya}. Let $\mathbb{Z}_m$ be the integers modulo $m$.

\begin{lemma}\label{lemma: p prime}
Let $p$ be a prime integer, $p>3$. If $2$ is a primitive root of $p$, then $C^+_{p-1}\cup C^+_1$ is a (modular) queen digraph and $C^+_{p-1}$ is a queen digraph. 
\end{lemma}

\begin{proof}
Let $D$ be the $1$-regular digraph with $V(D)=\mathbb{Z}_p$ defined by $(u,v)\in E(D)$ if and only if $v\equiv 2u \pmod p$. An easy check shows that the adjacency matrix of $D$ is a solution for the modular $p$-queens problem with a queen in position $(0,0).$
Since $2$ is a primitive root of $p$, the sequence, $1, 2,4, 8\pmod p,\ldots, 2^{p-1}\pmod p$ defines a labeling of a cycle of length $p-1$. Thus, $C^+_{p-1}\cup C^+_1$ is a (modular) queen digraph and $C^+_{p-1}$ is a queen digraph.\qed
\end{proof}

Park et al. characterized in \cite{ParParKim} some primes for which $2$ is a primitive root. 

\begin{theorem}\cite{ParParKim}
Let $p$ and $q$ be primes.
\begin{itemize}
\item Let $p=2q+1$. Then $2$ is a primitive root modulo $p$ if and only if $q\equiv 1\pmod 4$.
\item Let $p=4q+1$. Then  $2$ is a primitive root modulo $p$, for all $p$.
\end{itemize}
\end{theorem}

Let $p$ be a Mersenne prime, that is, a prime of the form $p=2^n-1$ for some integer $n$. It is known that the order of $2$ modulo $p$ is $n$. Thus, using a similar proof that the one of Lemma \ref{lemma: p prime}, we obtain the next result.

\begin{lemma}\label{lemma: p primeMersenne}
Let $p=2^n-1$ be a prime integer, $p>3$. Then $(p-1)/nC^+_{n}\cup C^+_1$ is a (modular) queen digraph and $(p-1)/nC^+_{n}$ is a queen digraph. 
\end{lemma}

The {\it Jacobsthal sequence} (or  {\it Jacobsthal numbers}) is an integer sequence which appears in \cite{Sloanweb} as `A001045'. It has connections with multiple applications, some of them can be found in \cite{Sloanweb}. Let $D_n$ be the $1$-regular digraph on $[1,n]$ defined by $(u,v)\in E(D_n)$ if and only if $v\equiv -2u+2 \pmod n$. The structure of $D_n$ was characterized in \cite{MJM_LMR} using the  previous sequence.

\begin{lemma}\cite{MJM_LMR}\label{lemma: structure D_n}
Let $n$ be an odd integer, $(a_i)$ the Jacobsthal sequence and $\Theta_k\subset \{1,2,\ldots, n\}$ such that $x\in \Theta_k$ if $k$ is the minimum $i$ with $
3a_ix\equiv2a_i\ \pmod n$. Then,
$$D_n\simeq \bigcup_{k=1}^n\frac{|\Theta_k|}k C^+_k.$$  
\end{lemma}

Using the solutions of the $n$-queens problem provided by Pauls, for $n\equiv 1,5 \pmod n$, in \cite{Pau1,Pau2}, and as a corollary of the above result we obtain the next lemma.
 
\begin{lemma} Let $n$ be an odd integer such that $n\equiv 1,5$ $\pmod 6$, $(a_i)$ the Jacobsthal sequence and $\Theta_k\subset \{1,2,\ldots, n\}$ such that $x\in \Theta_k$ if $k$ is the minimum $i$ with $
3a_ix\equiv2a_i\ \pmod n$. Then,
$$\bigcup_{k=1}^n\frac{|\Theta_k|}k C^+_k$$ is a queen digraph. 
\end{lemma}

\begin{proof} The adjacency matrix of $D_n$ is a $\pi/2$ radiants clockwise rotation of the solution of the $(n-1)\times (n-1)$-queens problem provided by Pauls, for $n\equiv 1,5 \pmod n$, in \cite{Pau1,Pau2}, when an extra queen is added in position (0,0). Thus,$D_n$ defines a queen labeled digraph. By Lemma \ref{lemma: structure D_n}, the result follows. \qed
\end{proof}

Alhough $C_3^+$ is not a queen graph, for every positive integer $m\ge 3$, $m\equiv 0,1 \pmod 3$, the union of $m(m-1)/3$ strong oriented cycles of length $3$ is a queen digraph. See the next proposition.
 
\begin{proposition}\cite{BloLamMunRiu}\label{propo: union-3-cycles}
Placing $m(m-1)$ queens in positions $(i,f(i))$, $i=1,2,\ldots, m(m-1)$ is a solution to the problem of queen labeling a union of $m(m-1)/3$ copies of $C_3^+$, where $f(i)=(m-1)[(i-1) \pmod{m}]+(m-1)-\lfloor (i-1)/m\rfloor$.
\end{proposition}

A known result in the area of graph products (see for instance, \cite{HamImrKlav11}) is that the direct product of two strongly oriented cycles produces copies of a strongly oriented cycle, namely,
\begin{equation}\label{producte_directe_strong_cicles}
    C_m^+ \otimes C_n^+\cong \gcd (m,n)C_{\hbox{lcm}(m,n)}^+.
\end{equation}

By combining Proposition \ref{propo: union-3-cycles}, Corollary \ref{coro: queen_and_product_of_$1$-regular} and (\ref{producte_directe_strong_cicles}) we obtain the next lemma.

\begin{lemma}
Let $m$ be an integer with $m\ge 3$, $m\equiv 0,1 \pmod 3$ and $p$ a prime $p>4$ such that $2$ is a primitive root of $p$. Then,
$m(m-1)/3$ copies of $\gcd((p-1,3)C_{{\rm{lcm}}(p-1,3)}^+\cup C_3^+$ is a queen digraph.
\end{lemma}

\begin{proof}
Recall that, by Proposition \ref{propo: union-3-cycles}, $(m(m-1)/3)C_3^+$ is a queen digraph and by Lemma \ref{lemma: p prime}, $C^+_{p-1}\cup C^+_1$ is a modular queen digraph. Thus, in particular, the induced sums and differences of $C^+_{p-1}\cup C^+_1$ satisfy the statement of Corollary \ref{coro: queen_and_product_of_$1$-regular}. Hence, by Corollary \ref{coro: queen_and_product_of_$1$-regular}, the direct product $(m(m-1)/3)C_3^+\otimes (C^+_{p-1}\cup C^+_1) $ is a $1$-regular queen digraph. Note that, for any pair of digraphs $D_1$ and $D_2$, $D_1\otimes D_2\simeq D_2\otimes D_1$. Therefore,
$$(m(m-1)/3)C_3^+\otimes (C^+_{p-1}\cup C^+_1) \simeq (m(m-1)/3)\gcd(p-1,3)C_{\hbox{lcm}(p-1,3)}^+\cup C_3^+,$$ and the result follows. \qed
\end{proof}


\end{document}